\documentclass{amsart}
\usepackage{amssymb}
\usepackage{cite}
\usepackage{mathabx} 
\usepackage{hyperref} 

\usepackage[T1]{fontenc} 
\usepackage[utf8]{inputenc} 

\usepackage{color,soul}
\definecolor{ItalianApricot}{rgb}{1,0.7,0.5}
\sethlcolor{ItalianApricot}

\theoremstyle{plain}
\newtheorem{thm}{Theorem}[section]

\newtheorem{lem}[thm]{Lemma}

\theoremstyle{definition}
\newtheorem{defn}[thm]{Definition}

\theoremstyle{remark}
\newtheorem{remark}[thm]{Remark}

\numberwithin{equation}{section}

\renewcommand{\epsilon}{\varepsilon}
\renewcommand{\phi}{\varphi}

\begin{document}

\title[Indifference for genericity: Union and subsequence sets]{Notions of indifference for genericity:\\ Union sets and subsequence sets}

\date{\today}

\author{Tejas Bhojraj}
\address[Tejas Bhojraj]{Department of Mathematics, University of Wisconsin--Madison, 480 Lincoln Dr., Madison, WI 53706, USA}
\email{bhojraj@math.wisc.edu}


\maketitle

\begin{abstract}
A set $I$ is said to be \emph{a universal indifferent set for $1$-genericity} if for every \emph{$1$-generic} $G$ and for all $X \subseteq I$, $G \Delta X$ is also $1$-generic. Miller \cite{day2013indifferent} showed that there is no infinite universal indifferent set for $1$-genericity. We introduce two variants (union and subsequence sets for $1$-genericity) of the notion of universal indifference and prove that there are no non-trivial universal sets for $1$-genericity with respect to these notions. In contrast, we show that there \emph{is} a non-computable subsequence set for weak-$1$-genericity. 
\end{abstract}

\section{Overview}
The $1$-generic and weak-$1$-generic sets are important classes of sets in computability theory\cite{downey2010algorithmic}. A set $I$ is said to be \emph{a universal indifferent set for $1$-genericity} if for every $1$-generic $G$ and for all $ X \subseteq I$, the symmetric difference of $G$ and $X$ ($G \Delta X$) is also $1$-generic \cite{day2013indifferent}. We can similarly define a \emph{universal indifferent set for weak-$1$-genericity}. Figueira, Miller and Nies coined the term `indifferent' and examined indifferent sets for random sequences\cite{figueira2009indifferent}. Indifferent sets for $1$-genericity were examined by Day\cite{day2013indifferent}. See \cite{day2013indifferent} for a comprehensive survey of the literature investigating indifferent sets for various classes of sets. It is easy to check that any finite set is a universal indifferent set for $1$-genericity. Miller showed that there is no infinite universal indifferent set for $1$-genericity while Day showed that there \emph{is} a non-computable universal indifferent set for weak-$1$-genericity\cite{day2013indifferent}. A natural question arises: Is there a notion of indifference for $1$-genericity for which there \emph{is} a non-trivial example?

We consider two possibilities: (1) a \emph{union set for $1$-genericity} and (2) a \emph{subsequence set for $1$-genericity}. Union sets and subsequence sets for weak-$1$-genericity are defined similarly. While being a union set for $1$-genericity (weak-$1$-genericity) is a weaker notion than being a universal indifferent set for $1$-genericity (weak-$1$-genericity), being a subsequence set for $1$-genericity (weak-$1$-genericity) is not obviously comparable to being a universal indifferent set for $1$-genericity (weak-$1$-genericity). It is easy to see that any finite set is a union set and that any computable set is a subsequence set for 1-genericity (weak-$1$-genericity). 

We show: (1) that there is no infinite union set for $1$-genericity, (2) that there is no infinite subsequence set for $1$-genericity and (3) that there \emph{is} a non-computable subsequence set for weak-$1$-genericity. It is implicit in the work of Day that there is a non-computable union set for weak-$1$-genericity. Previously, Kuyper and Miller \cite{JR} have attempted to weaken the notion of universal indifference and to see if there is a non-computable universal set for $1$-genericity with respect to the weaker notion. They introduced the notion of \emph{$1$-generic stabilising sets}, using symmetric difference in place of union, and proved that there is no non-computable $1$-generic stabilising set \cite{JR}. In their proof, they used the Jockusch\textendash Shore modification of the Posner\textendash Robinson theorem \cite{jockusch1983pseudojump}, a technique which also inspired the proofs in this paper. 

\section {Background and notation}
We assume some familiarity with basic computability theory and with the definitions of $1$-generic, weak-$1$-generic, immune, and hyper-immune, all of which may be found in \cite{downey2010algorithmic} and \cite{odifreddi1999classical}. In what follows, we will use two results of Jockusch on numerous occasions: \\(1) If $X \subset \omega$ is hyperimmune, then there is a $1$-generic $G\supseteq X$.\\(2) A $1$-generic set cannot compute the halting problem ($\emptyset^{'}$). \\The result (1) appears as Proposition 4.7 in the paper by Hirschfeldt, Jockusch, Kjos-Hanssen, Lempp, and Slaman \cite{hirschfeldt2008strength} while (2) appears in \cite{odifreddi1999classical} as Proposition XI.2.3.

For any $S \subseteq \omega$, we abuse notation and denote the characteristic function of $S$ by $S$. I.e., for all $ n \in \omega$, $ S(n)=1 \iff n\in S$. This allows us to identify subsets of $\omega$ with elements of $2^{\omega}$. For any $S \subseteq \omega $, $ \overline{S}$ is the complement of $S$. For $\tau$ and $\sigma \in 2^{<\omega}$ and $A \in 2^{\omega}$,  $\tau \succeq \sigma$ and $A \succeq \sigma$ will denote that $\sigma$ is an initial segment of $\tau$ and \emph{A} respectively. Similarly, $\succ$ will denote ``strict initial segment''.

\section{Union Sets}

\begin{defn}
A set $X$ is said to be a \emph{union set for $1$-genericity (weak-$1$-genericity)} if for all $1$-generic (weak-$1$-generic) sets $G$, $G\cup X$ is also $1$-generic (weak-$1$-generic).
\end{defn}

Clearly, if $X$ is a universal indifferent set for $1$-genericity (weak-$1$-genericity), then it is also a union set for $1$-genericity (weak-$1$-genericity). Hence, by a result of Day \cite{day2013indifferent}, there is a non-computable union set for weak-$1$-genericity. We now show that there is no infinite union set for $1$-genericity. It is easy to see that any finite set is a union set for $1$-genericity.

\begin{thm}\label{thm:three}
If $X$ is infinite, there is a $1$-generic $A$ such that $A \cup X$ is not $1$-generic. I.e., there is no infinite union set for $1$-genericity.
\end{thm}
\begin{proof}
Fix an infinite set $X$. If $X$ is $1$-generic, then it is easy to see that $\overline{X}$ is $1$-generic too. Since $\omega$ is not $1$-generic, we may take $A$ to be $\overline{X}$. So, let $X$ be non-$1$-generic. If $\overline{X}$ is hyper-immune, then there is a $1$-generic $G\supseteq \overline{X}$. So, $\overline{G}\subseteq {X}$. Then $A=\overline{G}$ is $1$-generic and $A\cup X = X$ is not $1$-generic, as desired. So, we may assume that $X$ is not co-hyperimmune. Lastly, since it is easy to see that a $1$-generic cannot contain an infinite c.e.\ set, we may assume that $X$ is immune. (If $X$ contains an infinite c.e.\ set, then for any $G$,  $G \cup X$ is not $1$-generic as it contains an infinite c.e.\ set.)\\
\emph{Construction}: Since $X$ is not co-hyperimmune, let \emph{f} be computable and strictly increasing so that for every \emph{n}, there exists an $m \in (f(n),f(n+1))$ with $X(m) = 0$. We build $A$ by initial segments. Suppose that at stage $s+1$, $\sigma_{s}$ has been built and $|\sigma_{s}|=f(m)$ for some \emph{m}. Consider the c.e.\ set:
\[V= \{n>f(m):\; \exists \tau \succeq \sigma_{s}0^{f(n)-f(m)}1^{f(n+1)-f(n)},\;\;\tau \in W_{s}\}.\]
Case 1. $|V|= \infty$. Pick an $n\in V\backslash X$ and the least $\tau$ in a canonical enumeration of $W_s$ witnessing $n \in V$. Let $x= \mu z[ f(z)>|\tau|]$. 
\\If $s\notin \emptyset^{'}$, let $\sigma_{s+1}=\tau 1^{f(x)-|\tau|}0^{f(x+1)-f(x)}$.
If $s \in \emptyset^{'}$, let $\sigma_{s+1}=\tau 1^{f(x+1)-|\tau|} .$
\\Case 2. $|V|<\infty$. Pick an $n\in X\backslash V, n>f(m)$. \\If $s\notin \emptyset^{'}$, let $\sigma_{s+1}= \sigma_{s}0^{f(n)-f(m)}1^{f(n+1)-f(n)}0^{f(n+2)-f(n+1)}$.\\ If $s\in \emptyset^{'}$, let $\sigma_{s+1}= \sigma_{s}0^{f(n)-f(m)}1^{f(n+2)-f(n)}$.\\ This completes the construction.
Note: In both cases, we have $n>f(m)$ and $\sigma_{s}0^{f(n)-f(m)}\prec A$. So, $ |\sigma_{s}|= f(m)<n<f(n)= |\sigma_{s}0^{f(n)-f(m)}|$ ensures that $n\notin A.$
\item
\emph{Verification}: As a $1$-generic set cannot compute $\emptyset^{'}$, it suffices to prove the following claims. \\
\emph{Claim}: The set $A$ is $1$-generic. \\\emph{Proof}: Take stage $s+1$. Since $X$ is immune, the appropriate \emph{n} as in the construction can be found in both cases. In case (1), there exists a $ \tau \prec A$ such that $\tau \in W_{s}$. In case (2), $n\notin V$ and $n>f(m)$. So, for every $\tau \succeq \sigma_{s}0^{f(n)-f(m)}1^{f(n+1)-f(n)}, \tau \notin W_{s}$. We ensure that $\sigma_{s}0^{f(n)-f(m)}1^{f(n+1)-f(n)} \prec A$. \qed\\
\emph{Claim}: $A\cup X$ computes $\emptyset^{'}$.\\
\emph{Proof}:
We describe the algorithm used by $B=A\cup X$ to find $\sigma_{s+1}$ from $\sigma_{s}$. First, find \emph{m} from $|\sigma_{s}|$. To find the \emph{n} such that $\sigma_{s}0^{f(n)-f(m)}1^{f(n+1)-f(n)} \prec \sigma_{s+1}$, find the first block in $B\upharpoonright(f(m),\infty)$ of the form $1^{f(x+1)-f(x)}$ for some $x>f(m)$. More precisely, $B$ finds $x$ where \[x=\mu p> f(m)\;\;[\;\forall y\in (f(p), f(p+1)), B(y)=1].\] We see that $x=n$. Indeed, for every $y\in [f(m), f(n))$, we have that $A(y)=0$ and for every $ k\in [1,n-m]$, there exists $y \in (f(m+k-1), f(m+k))$ with $X(y)=0$. So, for every $k\in [1,n-m]$, there exists $ y \in (f(m+k-1), f(m+k))$ with $B(y)=0$. Finally, by construction, we have that for every $y\in (f(n), f(n+1)), B(y)=1$.\\So, $x=n$.\\If $n\notin B$, find (computably) the $\tau$ and $x$ as in case 1. By similar reasoning as above, if there exists a $ y\in (f(x),f(x+1))$ with $B(y)=0$, then $s\notin \emptyset^{'}$ and if for every $y\in (f(x),f(x+1))$ we have that $ B(y)=1$, then $s\in \emptyset^{'}$. By checking which case holds, $B$ can find $\sigma_{s+1}$.
If $n\in B$, then by the note, it must be that $n\in X$. So, case 2 occurred. Again, by checking whether or not there exists a $y\in (f(n+1),f(n+2))$ with $ B(y)=0$, $B$ finds $\sigma_{s+1}$ and determines if $s\in \emptyset'$. \qed
\item
Theorem ~\ref{thm:three} is proved.
\end{proof}

\section{Subsequence Sets}
\begin{defn}
If $I \subseteq \omega$ is an infinite set and $ i_{0}<i_{1}<i_{2}\dots$ is the listing of $I$ in increasing order, $P_I$ (the \emph{principal function of I}) is defined by $P_{I}(n)=i_{n}$ for all $n \in \omega$.
\end{defn}
\begin{defn}
Given $A,I \subseteq \omega$,  $A_I$ (\emph{$I$ subsequence of $A$}) is the set defined by $A_I(n)= A(P_{I}(n))$ for all $n \in \omega$.
\end{defn}
\begin{remark}
The idea is as follows: $A_I$ is the subsequence obtained by considering a bit of the $A$ sequence only if the $I$ sequence contains a one in that location. Alternatively, $A_I$ is the subsequence obtained from the $A$ sequence by deleting a bit if the $I$ sequence contains a zero in that location. Of course, we are thinking of sets as elements of $2^{\omega}$.
\end{remark}
\begin{defn}
 We say that $I\subseteq \omega $ is a \emph {subsequence set for $1$-genericity}, if for all $G\subseteq \omega$, if $G$ is $1$-generic then $G_I$ is also $1$-generic.
\end{defn}

It is easy to see that any computable set is a subsequence set for $1$-genericity.

\begin{defn}
If $G$ is $1$-generic, say that $I\subseteq \omega $ is a \emph {subsequence set for G} if $G_I$ is a $1$-generic set.
\end{defn}
\begin{remark}
It can be shown that given a countable collection of $1$-generic sets $(G_s)_{s \in \omega}$, there is a non-computable $I$ which is a subsequence set for each $G_s$. We omit the proof as this result is not relevant here.
\end{remark}  
\begin{thm}\label{thm:one}
There is no non-computable subsequence set for $1$-genericity.
\end{thm}

Given a hyperimmune set $H$, there is a $1$-generic $G$ containing it and hence such that $G_H = \omega$. So, a subsequence set for $1$-genericity cannot be hyperimmune. So, Lemma ~\ref{lem:main} suffices to prove Theorem ~\ref{thm:one}.

\begin{lem} \label{lem:main}
There is no non-computable, non-hyperimmune subsequence set for $1$-genericity.
\end{lem}

We will prove the lemma in two cases, for which we need two definitions. 
\begin{defn}\label{defn:main}
A non-computable set $I\subseteq \omega $ is \emph {wild} if there exists a computable, strictly increasing \emph{f} such that for every $ n , I\cap (f(n),f(n+1)) \neq \emptyset$ and $h$ given by $h(n):= |I\cap (f(n),f(n+1))|$ is non-computable.\\A non-computable set $I\subseteq \omega $ is \emph {tame} if there exists a computable, strictly increasing \emph{f} such that for every $n, I\cap (f(n),f(n+1)) \neq \emptyset$ and \emph{h} given by $h(n):= |I\cap (f(n),f(n+1))|$ is computable.
\end{defn}
 
From the definition of hyperimmune (See Definition 2.1 on page 80 in \cite{Soare:1987:RES:22895}), it is easy to see that non-hyperimmune sets are either tame or wild. Also note that a set can be both tame and wild at the same time. We prove Lemma ~\ref{lem:main} separately for tame and wild sets.
\subsection{The proof of Lemma ~\ref{lem:main} for wild sets}

\begin{thm}\label{thm:two} There is no wild subsequence set for $1$-genericity.
\begin{proof}

Let $I$ be wild and let \emph{f} and $h$ be as in Definition ~\ref{defn:main}.
We first introduce some notation: For all $\sigma, \tau \in 2^{<\omega}$ and $ n\in \omega$, $\sigma \tau^{n}$ will denote $\sigma$ followed by \emph{n} copies of $\tau$. For every $\tau \in 2^{<\omega}$ and every $ m,p \in \omega$, if $|\tau|= f(m)$, $i_1,i_2,\dots,i_p \in \{0,1\}$, and $n_1,n_2,\dots,n_p \in \omega $, then the string \[\sigma= \tau i_{1}^{f(m+n_1)-f(m)}i_{2}^{f(m+n_2+n_1)-f(m+n_1)} \cdots i_{p}^{f(m+n_p+ \cdots +n_1)- f(m+n_{p-1}+ \cdots +n_1)}\] will be denoted by $\tau [i_1]^{n_1}[i_2]^{n_2} \cdots [i_p]^{n_p}$. If $f(m+n_j+n_{j-1}+ \cdots +n_1)<x<f(m+n_{j+1}+n_j+ \cdots +n_1)$, for some $0<j<p$, we say that ``$x$ is in $[i_{j+1}]$ in $\sigma$'', or just ``$x$ is in $[i_{j+1}]$'', when $\sigma$ is clear from context. For any $n$, the interval $[f(n-1),f(n))$ will be called a ``block''. For $i= 1$ or $0$, $[i]$ may be thought of as a ``block'' of $f(t+1)-f(t)$ many $i$s, occupying the bits $f(t), f(t)+1, \dots f(t+1)$ for some \emph{t}. I.e., $[i]$ will denote $i^{f(t+1)-f(t)}$ for a \emph{t} depending on the location of $[i]$ in the string. For example, if $|\sigma|=f(m)$, then in the string $\sigma [1][0]= \sigma 1^{f(m+1)-f(m)} 0^{f(m+2)-f(m+1)} $, the $[1]$ denotes the $1^{f(m+1)-f(m)}$ ``part'' and the $ [0]$ denotes the $ 0^{f(m+2)-f(m+1)}$ ``part''.\medskip

We build a $1$-generic $A$ by initial segments such that $A_I$ is not $1$-generic. Let $\{W_s\}_{s \in \omega}$ be a list of all c.e.\ sets of strings that are closed under taking extensions. Note that to make $A$ $1$-generic, it is enough to meet or avoid each such set. \medskip

\emph{Construction}:
Suppose that at stage $s+1$ we have built $\sigma_s \in 2^{<\omega}$ and \emph{m} is such that $|\sigma_s|=f(m)$. We build $\sigma_{s+1}$ as follows.
Consider the c.e.\ set:
\[V= \{n: \lfloor (n-1)/2\rfloor>f(m),\; \exists x\;\exists \tau \succeq \sigma_{s}([0][1])^{n}[1]^{f(m+2n+1)}[0],\;|\tau|=f(x),\;\tau \in W_{s}\}.\]

Search for an \emph{n} with $ n-1 > f(m)$ for which 1 or 2 holds;\\
 1. \{$2n+1 \in V$ and $P_{I}(n)$ is in a [0]\} or \{$2n+1 \notin V$ and $P_{I}(n)$ is in a [1]\}
\\2. \{$2n \in V$ and $P_{I}(n)$ is in a [1]\} or \{$2n \notin V$ and $P_{I}(n)$ is in a [0]\}.
\begin{remark}\label{rem:oneee}
In the above, ``$P_{I}(n)$ is in a [0] (or in a [1]) '' means that $P_{I}(n)$ is in a [0] (or in a [1]) in $ \sigma_{s}([0][1])^{n}[1]^{f(m+2n+1)}[0]$. Note that $n-1>f(m)$  implies that $ \lfloor (2n-1)/2 \rfloor,\lfloor (2n+1-1)/2 \rfloor > f(m) $ and that $ |\sigma_{s}|=f(m)< n-1 < n \leq P_{I}(n)<P_{I}(2n)<|\sigma_{s}([0][1])^{n}|$. The last inequality holds since at least one element of $I$ occurs in each ``block'' in $([0][1])^{n}$. So, for the \emph{n}s we search over, $P_I(n)$ is in a $[1]$ or in a $[0]$ in the $\sigma_{s}([0][1])^{n}$ ``part'' of $\sigma_{s}([0][1])^{n}[1]^{f(m+2n+1)}[0]$. \end{remark}
If 2 holds and $2n\notin V$, let $\rho = \sigma_{s}([0][1])^{2n}[1]^{f(m+4n+1)}[0]$. If $2n\in V$, take the first $\tau, x$ (i.e., the first to appear in a canonical enumeration of $W_s$) witnessing this. Let
\begin{align}
\label{eq:1}
    \rho=\tau [1]^{f(x)}[0].
\end{align}
So, in both cases, there is a $y$ such that $|\rho|=f(y)$. If 1 holds, find $\rho$ similarly.\\ Let $\sigma_{s+1}=\rho[1][\emptyset'(s)]^{f(f(y)+1)}[1]$. This ends the construction.\\
\medskip
\emph{Verification}:
We need the following lemmas.
\begin{lem}
The search in the construction at stage \emph{s} halts.
\end{lem}  
\begin{proof}
Suppose not. Then, there is a $N= f(m)+1$ such that for all $ n>N$, $2n+1 \in V \iff P_{I}(n)$ is in a $ [1]$ and $2n\in V \iff P_{I}(n)$ is in a $ [0]$. Since for all $n, P_{I}(n)$ is in a $[1]$ or in a $[0]$ but not both, we have that for almost every $ n,\; 2n\in V$ or $2n+1\in V$ but not both. Note that $V$ is c.e. So, for almost every \emph{n}, by waiting for $2n$ or $2n+1$ to appear in $V$, we can compute if $ P_{I}(n) $ is in a [0] or in a [1]. $P_{I}$ is strictly increasing and for all $n , I\cap (f(n),f(n+1)) \neq \emptyset$. So, $h$ can be computed by successively checking in increasing order (i.e., for $N+1,N+2, \dots n,n+1,\dots$) if $ P_{I}(n)$ is in a $ [0]$ or in a $[1]$. It is important for $P_I(n)$ to be in a $[1]$ or in a $[0]$ in the $([0][1])^{n}$ ``part'' of $\sigma_{s}([0][1])^{n}[1]^{f(m+2n+1)}[0]$ and that it is not in the $\sigma_s$ or the $[1]^{f(m+2n+1)}[0]$ parts. Remark ~\ref{rem:oneee} guarantees this.
\end{proof}
\begin{lem} The set $A_I$ computes $\emptyset^{'}$ (and hence it is not $1$-generic).\end{lem}
\begin{proof}
First, we introduce some notation. For $\tau \in 2^{<\omega}$, define $\tau_I \in 2^{<\omega}$ by $\tau_I(n)= \tau (P_{I}(n))$ for \emph{n} with $P_{I}(n)\leq |\tau|$. So, note that $|\tau_I|= \max\{n: P_I(n) \leq |\tau|\}$. For $i=0$ or $1$, let $\{i\}^{n}$ denote $i^{m}$ for some $m\geq n$. Let $\{i\}^{1}=\{i\}$. Since for all $n, I\cap (f(n),f(n+1)) \neq \emptyset$, we have that for every $ \sigma \in 2^{<\omega}$, $ (\sigma[i])_{I}=\sigma_I \{i\} $. \\ To prove the lemma, we show that $A_I$ finds $\sigma_{s+1}$ and $(\sigma_{s+1})_I$ from $\sigma_{s}$ and $(\sigma_{s})_I$. The idea is: as $I$ has at least one element in each ``block'', the number of alternations of [0] and [1] in $A$ is coded as the number of alternations of $\{0\}$ and $\{1\}$ in $A_I$. So, $A_I$ can run the following procedure. Find \emph{m} such that $|\sigma_{s}|=f(m)$ and $k$ such that $(\sigma_{s}) _I(\{0\}\{1\})^{k}\{1\}^{f(m+2k+1)}\{0\} \prec A_I$.  
Suppose $k=2n$. Check if $n\in A_I$ or not. If yes, then $2n \in V$ must hold and so compute the least $\tau, x $ witnessing this with $|\tau|=f(x)$. So, the  $\rho=\tau [1]^{f(x)}[0]$ is found. Now, since $\rho$ ends in a $[0]$ and $\rho[1]\prec A$, $\rho_I$ ends in a $\{0\}$ and $\rho_I \{1\} \prec A_I$. So, $A_I$ finds $\rho_I$ by finding $\delta$ such that \[(\sigma_{s})_ {I}(\{0\}\{1\})^{k}\{1\}^{f(m+2k+1)}\{0\} \; \prec \;\delta \{1\}^{f(x)}\{0\}\{1\}\; \prec A_I.\]

 It must be that $\rho_I = \delta \{1\}^{f(x)}\{0\}\;$ and $\delta = \tau_{I}$. 
 
 (The details are: $A_I$ finds $\delta$ as follows: First, find $$M= |(\sigma_{s})_ I(\{0\}\{1\})^{k}\{1\}^{f(m+2k+1)}\{0\}|$$ such that $(\sigma_{s})_ I(\{0\}\{1\})^{k}\{1\}^{f(m+2k+1)}\{0\}1 \prec A_I$. Then, find the first place in $A_I\upharpoonright[M, \infty)$ where $\{1\}^{f(x)}$ occurs. As $\{1\}^{f(x)}$ cannot occur in $\tau_ I$ (since $|\tau|=f(x)$), $\delta = \tau_{I}$ is found correctly. Although $\tau_I \prec A_I$, $A_I$ cannot compute $\tau_ I$ without a signal saying ``$\tau_I$ ends here''. The $[1]^{f(x)}[0]$ in equation \ref{eq:1}, which cannot possibly occur within $\tau$, signals the end of $\tau$.)

Now, let $n \notin A_I$. So, $2n \notin V$ must be true. Then, $\rho$ and $f(y)=|\rho|$ can be found computably. We have that $\rho_I=(\sigma_{s})_ {I}(\{0\}\{1\})^{k}\{1\}^{f(m+2k+1)}\{0\}$. So, $A_I$ can find $\rho_I$, as $\{1\}^{f(m+2k+1)}$ cannot occur in $(\sigma_{s})_ {I}(\{0\}\{1\})^{k} \prec A_I$ and since $\rho_I$ ends in a \{0\} and $\rho_I\{1\}\;\prec A_I$. So, in both cases, $\rho$ and $y$ such that$ |\rho|=f(y)$ and $\rho_I$ are found. If $k$ is odd, proceed similarly.

The set $A_I$ can now find $i$ such that $\rho_I \{1\}\{i\}^{f(f(y)+1)}\{1\}\{0\} \prec A_I$ since $\{i\}^{f(f(y)+1)}$ cannot occur in $\rho_I$. The string $ \sigma_{s+1} $ can now be found computably and it only remains to show that $(\sigma_{s+1})_I$ can be found. So far, $A_I$ has computed $\lambda = \rho[1][i]^{f(f(y)+1)}[1]$. By construction, $\sigma_{s+1}= \lambda [1]$ and $\lambda[1][0]=\sigma_{s+1}[0] \prec A $ and so, $\lambda_I \{1\}\{0\}=(\sigma_{s+1})_I \{0\} \prec A_I $. So, $A_I$ can find $(\sigma_{s+1})_I$. (In words: since $(\sigma_{s+1})_I$ ends in a $\{1\}$ and $(\sigma_{s+1})_I\{0\}\prec A_I$, $(\sigma_{s+1})_I$ can be found correctly by $A_I$.)
\end{proof}
\begin{lem}
The set A is $1$-generic.
\end{lem}
\begin{proof}
By the previous lemma, the search at each stage of the construction halts. So, $A$ meets or avoids each $W_s$, as by assumption, each $W_s$ is closed under taking extensions.
\end{proof}
These lemmas suffice to prove Theorem ~\ref{thm:two}.
\end{proof}

\end{thm}
 
\subsection{The proof of Lemma ~\ref{lem:main} for tame sets}

\begin{thm}\label{thm:sub} There is no tame subsequence set for $1$-genericity.
\begin{proof}

 Let $I$ be tame and let \emph{f} and \emph{h} be as in Definition ~\ref{defn:main}. Let, without loss of generality, $h>1$. Except for the usage of $\{\}$, the proof of Theorem ~\ref{thm:sub} uses the same notation as in the proof of Theorem ~\ref{thm:two}. If for some \emph{m} and $(\tau_{i})_{i\leq m} \subseteq 2^{<\omega}$, $\rho= \sigma \tau_{1}\tau_{2}\dots \tau_{m}$, then $\rho_I$ is denoted by $\sigma_I \{\tau_{1}\}\{\tau_{2}\} \dots \{\tau_{m}\}$. Call $\{\sigma\}$ `homogeneous' ($H$) if it has only $0s$ in it. For better readability, $\{[i]\}$ will be denoted by just $\{i\}$. The interval $[f(n-1),f(n))$ will be called the $n^{th}$ block.

\emph{Construction}: We build $A$, a $1$-generic, using an initial segments construction. At stage $= s+1$, let $\sigma_{s}$ have been built with $|\sigma_{s}|=f(m)$ for some \emph{m}. A string $\sigma$ is said to be \emph{suitable} for \emph{n} (at stage \emph{s}) if it has exactly one $1$ in it and $\sigma$ occupies the $(m+n+2)^{nd}$ block in $\sigma_{s}[0]^{n}[1]\sigma$. I.e., $|\sigma|=f(m+n+2)-f(m+n+1)$. Build $\sigma_{s+1}$ as follows:\\ Consider the c.e.\ set:
\[V= \{(n,\sigma) : \exists x\; \exists \tau \succeq \sigma_{s}[0]^{n}[1]\sigma [0]^{q}[1];\;\sigma(q)=1;\; \sigma \mbox{ is suitable for n}; \;\tau \in W_{s}; \;|\tau|=f(x)\}.\]

Check if $\exists n, \sigma$ with $\sigma$ suitable for \emph{n} such that 1 or 2 holds.\\
1. $(n,\sigma) \in V$ and the $\{\sigma\} $ in $(\sigma_{s})_I \{0\} ^{n}\{1\}\{\sigma\} \{0\} ^{q}$ is not $ H$.\\ 
2. $(n,\sigma) \notin V$ and the $\{\sigma\} $ in $(\sigma_{s})_I\{0\} ^{n}\{1\} \{\sigma\} \{0\} ^{q}$ is $ H$.\\
If $(n,\sigma) \in V$, take the least $\tau$ witnessing this. So, $|\tau|=f(x)$ for some $x$. Let $\rho = \tau$. If $(n,\sigma) \notin V$, let $\rho =\sigma_{s}[0]^{n}[1]\sigma [0]^{q}[1]$. Let $\sigma_{s+1}=\rho[\emptyset^{'}(s)]$. This completes the construction.
\\

\emph{Verification}:
To prove Theorem ~\ref{thm:sub}, it suffices to prove the following three lemmas. 

\begin{lem}

 For all stages \emph{s}, there exist $\sigma$ and n, with $\sigma$ suitable for n such that 1 or 2 holds.
 \end{lem}
 \begin{proof} Towards a contradiction, suppose that the lemma fails at stage \emph{s}. In what follows, ``suitable for $(.)$'' means ``suitable for $(.)$ at \emph{s}''. Then, for every \emph{n} and $ \sigma$ such that $\sigma$ is suitable for \emph{n}, we have
 
 \begin{align}
 \label{eq:2}
     (n,\sigma) \in V \iff \mbox{the }\; \{\sigma\} \mbox{ in  } (\sigma_{s})_I\{0\} ^{n}\{1\} \{\sigma\} \mbox{ is}\;H.
 \end{align}
Let $ f(m)=|\sigma_{s}|$. We show that for all $ x>m+2$, $I\cap (f(x-1),f(x)]$ can be computed, a contradiction. Fix $ x>m+2$ and let $p= x- (m+2)$. So, note that if $\sigma$ is suitable for \emph{p}, then $|\sigma|= f(x)-f(x-1)$. As $h>1$, by the definition of being ``suitable for \emph{p}'', we have that for every $ \sigma $ suitable for \emph{p}, if $ z \in (f(x-1),f(x)]$ and $\sigma(z-f(x-1))=1 $, then
\begin{align}
\label{eq:3}
   \{\sigma\} \mbox{ in } (\sigma_{s})_I\{0\} ^{p}\{1\} \{\sigma\}  \mbox{ is $H$} \iff I(z)=0. 
\end{align}
Also, by the definitions of $h$ and of being  ``suitable for \emph{p}'' , there are $f(x)-f(x-1)-h(x)$ many strings $\sigma$ such that $\sigma$  is  ``suitable for \emph{p}'' and such that $\exists z \in (f(x-1),f(x)]$ with $\sigma(z-f(x-1))=1$ and $I(z)=0$. So, by \eqref{eq:3}, $f(x)-f(x-1)-h(x)$ many strings $\sigma$ are such that the $\{\sigma\} $ in $(\sigma_{s})_I\{0\} ^{p}\{1\} \{\sigma\} $ is $H$. So, by \eqref{eq:2}, if $V_p:=\{\sigma:(p,\sigma) \in V\}$, then $|V_p|= f(x)-f(x-1)-h(x)$. By \eqref{eq:2} and \eqref{eq:3}, $I\cap (f(x-1),f(x)]$ can be found from $V_{p}$. To find $V_p$, enumerate it until $f(x)-f(x-1)-h(x)$ many elements have appeared. Now, compute $I\cap (f(x-1),f(x)]$. Note that $x$ was arbitrary and that $V_p$ can be enumerated uniformly in $p=x-(m+2)$. This shows that $I$ is computable. 
 \end{proof}

\begin{lem}
The set $A$ is $1$-generic.
\end{lem}
\begin{proof}

By the previous lemma, the search halts at every stage. Also, without loss of generality, we may assume that for every $ s,\; W_s$ is closed under taking extensions. So, it does not matter that in the definition of $V$, we only search for strings $\tau$ such that $|\tau|=f(x)$ for some $x$.\end{proof} 
\begin{lem}
The set $A_I$ computes $\emptyset^{'}$ (and hence is not $1$-generic).
\end{lem}
\begin{proof}

Note: for every $ \sigma, \tau\in 2^{<\omega}$ and for every $a\in \omega$, if $|\tau|= f(a+1)-f(a)$ and $ |\sigma|= f(a)$, then the $\{\tau\}$ ``part'' of $\sigma_I\{\tau\}$ has length $|\{\tau\}|= h(a)$. In what follows, the strings $\rho$ and $\sigma$ refer to those appearing in the construction at stage $s+1$. We show that $A_I$ computes $\sigma_{s+1}$ and
$(\sigma_{s+1})_I$ given $\sigma_{s }$ and $(\sigma_{s})_I$ as follows. First, $A_I$ finds \emph{m} such that $|\sigma_{s }|=f(m)$. By the note and since \emph{h} is computable, using \emph{m}, $A_I$ finds \emph{n} and $\delta \in 2^{<\omega}$ such that $(\sigma_{s})_I\{0\}^{n}\{1\}\delta \prec A_I $
and $|\delta|= h(m+n+2)$. Then it must be that $\delta= \{\sigma\}$. Now, check if $\{\sigma\}$ is $ H$. If yes, then $\rho$ and $\rho_I $ are found computably in $A_I$ and \emph{h}. If no, then $A_I$ first finds the $\sigma$ using $(\sigma_{s+1})_I\{0\}^{n}\{1\}\{\sigma\} \{0\}^{q}\{1\} \prec A_I$. Having found the $\sigma$, $A_I$ can find the least $\tau$ and $x$ such that $|\tau|=f(x)$ as in the construction. Let $ M= \sum_{k\leq x} h(k)$. Then $A_I\upharpoonright M = \tau_I =\rho_I$. Now, $\sigma_{s+1}$ and $(\sigma_{s+1})_I$ are found computably in $A_I$.
\end{proof}

\end{proof}

\end{thm}
\begin{remark}
The proof of Lemma ~\ref{lem:main} shows that for any non-hyperimmune and non-computable set $I$, there is a 1-generic set $A$ such that $A_I$ computes $\emptyset'$. In particular, this implies the existence of a 1-generic set $A$ and a set $I$ such that $A_I$ computes $\emptyset'$.
\end{remark}

\section{A non-computable subsequence set for weak-1-genericity}

The proofs in this section are largely based on Section 3 of \cite{day2013indifferent}.
\begin{thm}\label{thm:subb}
There is a non-computable subsequence set for weak-$1$-genericity.
\end{thm}
\begin{proof}

 We need two definitions.
\begin{defn}
\label{defn}
If $S$ is a dense set in $2^{<\omega}$, then $l :\omega \rightarrow \omega$ and $t: \omega \rightarrow 2^{<\omega}$ are said to be \emph{``suitable for S''} if $(I)$ and $ (II)$ hold for all $k\in \omega $, with $\tau_k$ as in Lemma 3.1 in ~\cite{day2013indifferent}:
\\
\emph{(I)} \emph{l} and \emph{t} are defined inductively as follows: $l(0)=0,\;\; t(0)=\tau_0$,\\
$l(n+1)= l(n)+ |t(n)|,\;\; t(n+1)= \tau_{l(n+1)}$.\\
\emph{(II)} for every $ X \in 2^{\omega}$ and for every $ n,\;m\leq l(n) \Rightarrow (X\upharpoonright m)t(n)$
 meets $S$.\end{defn}
 
 The meaning of suitable in this section is independent of its meaning in the previous section.

 \begin{defn}
 If $D=(S_i)_{i\in\omega}$ is a collection of dense sets of strings, \emph{$D$ has $*$} if for all $i$ there exist $t_i,l_i$, suitable for $S_i$ such that for every $ n$, we have that $ T_{i,n} \in D$, where\\ $T_{i,n}= \{\sigma \tau \rho:\; \exists m>n,\; |\sigma|=l_{i}(2m),\;\tau=t_{i}(2m),\; \rho = t_{i}(2m+1)\}. $ 
 
 \end{defn}
 
 \begin{lem}\label{lem:onee}
If $D=(S_i)_{i\in\omega}$ has $*$, then there is a non-computable set $I$ such that for every $ X \in \mathcal{A}= \{Y\in 2^{\omega}: \forall i [Y$ meets $S_i]\},\;$we have that $X_I \in \mathcal{A}.$
 \end{lem}
 
 \begin{proof}
 Let $D=(S_i)_{i\in\omega}$ have $*$ as witnessed by $(t_i,l_i)_i$. Let $I$ be a set such that for every $ i$ for almost every $ n$, $|[l_{i}(2n), l_{i}(2n+2)) \cap \overline{I}| \leq 1 $. As there are continuum many sets with this property, but only countably many computable sets, $I$ can be taken to be non-computable. We show that $I$ works. Let $A\in \mathcal{A}$ and fix an arbitrary $i$. We must show that $A_I$ meets $S_i$. Now, since $D$ has $*$, $A$ meets $T_{i,n}$ for all \emph{n}. I.e., for every $ n$, there exists $m>n$ such that $\sigma \tau \rho \prec A$, and $|\sigma|=l_{i}(2m),\;\tau=t_{i}(2m),\; \rho = t_{i}(2m+1)$. So,
 \begin{align}
 \label{eq:5}
    \exists ^{\infty} n\text{ such that} (A\upharpoonright l_{i}(2n))t_{i}(2n)t_{i}(2n+1) \prec A.  
 \end{align}
By choice of $I$, for almost every $ n$, we have that
\begin{align}
\label{eq:6}
    I\upharpoonright [l_{i}(2n),l_{i}(2n+1)) = 1^{l_{i}(2n+1)- l_{i}(2n)}
\end{align}
or
\begin{align}
\label{eq:7}
    I\upharpoonright[l_{i}(2n+1),l_{i}(2n+2)) = 1^{l_{i}(2n+2)- l_{i}(2n+1)}
\end{align}
So, pick an \emph{n} such that \eqref{eq:5} and \eqref{eq:6} or \eqref{eq:5} and \eqref{eq:7} holds for \emph{n}. Assume that the former holds. So, $(A\upharpoonright l_{i}(2n))_I t_{i}(2n) \prec A_I$. As $M= |I\cap [0,l_{i}(2n))| \leq l_{i}(2n) $, we see that $|(A\upharpoonright l_{i}(2n))_I|=M \leq l_{i}(2n)$. So, by Definition \ref{defn}, (II), $(A\upharpoonright l_{i}(2n))_I t_{i}(2n)\prec A_I $ meets $S_i$. Similarly, $A_I$ meets $S_i$ if the latter holds.
 \end{proof}
 
 \begin{lem}
 There exist uniformly computable sequences of partial functions $(t_e)_{e\in \omega}$ and $(l_e)_{e\in \omega}$ such that if $W_e$ is a dense set of strings, then $t_e$ and $l_e$ are total and suitable for $W_e$ for all \emph{e}.
 \end{lem}
 \begin{proof}
 Fix \emph{e} and suppose that $l_{e}(n)$ and $t_{e}(n)$ have been found. Let $l_{e}(n+1)= l_{e}(n) + |t_{e}(n)|$ and list $2^{\leq l_{e}(n+1)}$ as $(\sigma_i)_{i<k}$ for some $k$. Search for $\rho_0 \succ \sigma_0, \rho_0 \in W_e$. If found, search for $ \rho_1 \succ \sigma_1 \rho_0,\;\rho_1 \in W_e$. Repeat this process for all $i < k$. I.e., if $(\rho_j)_{j<n}$ have been found, search for $\rho_n \succ \sigma_n \rho_0 \dots \rho_{n-1},\;\rho_n \in W_e$. Let $t(n+1)=\rho_0 \cdots \rho_{n}$. Clearly, $t_e$ and $l_e$ are uniformly computable in \emph{e} and are total and suitable for $W_e$ if $W_e$ is a dense set of strings.
 \end{proof}

Let $D$ be all dense c.e.\ sets of strings. So, $D= (W_{f(e)})_{e\in \omega}$ for some $f \leq \emptyset^{''}$. Let $S_e = W_{f(e)}$ for all \emph{e}. By the previous lemma and since $S_e$ is dense for all \emph{e}, $l_{f(e)}$ and $t_{f(e)}$ are total and computable for all \emph{e}. So, for all \emph{e} and \emph{n}, we have that $T_{e,n}= \{\sigma \tau \rho:\;\exists m>n,\; |\sigma|=l_{f(e)}(2m),\;\tau=t_{f(e)}(2m),\; \rho = t_{f(e)}(2m+1)\} $ is c.e.\ and dense. So, $T_{e,n} \in D$. Hence, $D$ has $*$.
So, by Lemma ~\ref{lem:onee}, Theorem ~\ref{thm:subb} is proved. 
\end{proof}

\section{Acknowledgement} 
I thank my PhD thesis advisor, Joseph S. Miller for introducing me to the topic dealt with in this paper and also for his constant encouragement, guidance and support.  

%

\bibliographystyle{plain}
\bibliography{references}

\end{document}